\documentclass{amsart}

\usepackage{amsmath, amsthm}
\usepackage{amssymb}
\usepackage{ esint }
\usepackage{comment}

\newtheorem{theorem}{Theorem}[section]
\newtheorem{lemma}[theorem]{Lemma}

\theoremstyle{definition}
\newtheorem{definition}[theorem]{Definition}

\theoremstyle{remark}
\newtheorem{remark}[theorem]{Remark}

\newcommand\R{\mathbb{R}}
\newcommand\N{\mathbb{N}}

\newcommand\E{\mathbb{E}}
\newcommand\B{\mathcal{B}}
\newcommand\supp{\mathrm{supp}}

\numberwithin{equation}{section}

\newcommand{\abs}[1]{\lvert#1\rvert}

\begin{document}

\title[Homogenization for the cubic NLS on $\mathbb R^2$]{Homogenization for the cubic nonlinear Schr\"odinger equation on $\mathbb R^2$}

\author{Maria Ntekoume}
\address{Department of Mathematics, University of California, Los Angeles, CA 90095, USA}
\email{mntekoume@math.ucla.edu}

%\date{June 2018.}

\begin{abstract}
We study the defocusing inhomogeneous mass-critical nonlinear Schr\"odinger equation on $\mathbb R^2$
$$i \partial_t u_n +\Delta u_n=g(nx) \abs{u_n}^2 u_n$$
for initial data in $L^2(\mathbb R^2)$. We obtain sufficient conditions on $g$ to ensure existence and uniqueness of global solutions for $n$ sufficiently large, as well as homogenization.
\end{abstract}

\maketitle

\section {Introduction}
We consider the Cauchy problem for the cubic nonlinear Schr\"odinger equation with inhomogeneous nonlinearity

\begin{equation}\label{eqn:NLS_n} \tag{$NLS_n$}
    \begin{cases}
    i \partial_t u_n +\Delta u_n=g(nx) \abs{u_n}^2 u_n\\
    u_n(0)=u_0 \in L^2(\R^2)
    \end{cases}
\end{equation}
and the homogeneous defocusing cubic nonlinear Schr\"odinger equation
\begin{equation}\label{eqn:NLS} \tag{$NLS$}
    \begin{cases}
    i \partial_t  u +\Delta u=\bar g  \abs{u}^2 u\\
    u(0)=u_0 \in L^2(\R^2)
    \end{cases}
\end{equation}
where $g\in L^\infty(\R^2)$ and $\bar g\geq 0$. If $\bar g=0$, (NLS) becomes the linear Schr\"odinger equation.

The (NLS) arises in various physical contexts in the description of nonlinear waves. For example, (NLS) models the propagation of intense continuous wave laser beams in a homogeneous Kerr medium, in which case the nonlinearity is generated as a result of the interaction of the electric field and the atoms of the dielectric medium. By comparison, the inhomogeneous ($NLS_n$) describes the propagation of laser beams  in an inhomogeneous medium; in this context, $n$ is proportional to the ratio between the scale of the laser beam and the spacing of the atoms of the medium.

In \cite{Merle}, Merle studied the inhomogeneous mass-critical NLS
$$iu_t+\Delta u=-g(x)|u|^{\frac 4 d}u,\,\,\,x\in \R^d $$
and obtained sufficient conditions on the coupling function $g$ to ensure the existence of blow-up solutions, as well as the nonexistence of minimal blow-up solutions in $L^2$. In \cite{RS}, Rapha\"el and Szeftel  discovered necessary and sufficient conditions on $g$ to ensure the existence and uniqueness of critical blow-up solutions for the same problem. In \cite{Farah} and \cite{Guzman}, the well-posedness of the inhomogeneous NLS was investigated for a specific family of coupling functions, namely for the problem
$$iu_t+\Delta u=\pm |x|^{-b}|u|^{\alpha}u,\,\,\,x\in \R^d $$
where $\alpha, b>0$. In \cite{CG}, Combet and Genoud established a classification of minimal blow-up solutions of this problem with $\alpha=\frac{4-2b}{d}$. The blow-up of solutions of mass-critical NLS with time-oscillating nonlinearity is studied in \cite{O} and \cite{ZZ}.

The question we will consider in this paper is that of homogenization for this problem. Homogenization problems have received a lot of attention and one of the most popular methods is the spectral approach based on the Floquet-Bloch theory. This method is used in \cite{Suslina} to study the behavior of the solution $u_n$ of the Cauchy problem for the Schr\"odinger-type equation 
$$i \partial _t u_n =\mathcal A_n u_n +F,$$
where $\mathcal A_n$ is a self-adjoint strongly elliptic second order differential operator with periodic coefficients depending on $nx$. For a special case of this problem, namely the Schr\"odinger equation with large periodic potential 
$$i \partial _t u_n -\nabla \cdotp [g(nx)\nabla u_n]+(n^2 c(nx)+d(x,nx))u_n=0,$$
where $g(y)$, $c(y)$, $d(x,y)$ are real-valued bounded functions defined for $x\in\R^d$, $y\in \mathbb T^d$, and $g(y)$ is symmetric and uniformly positive definite, homogenization was obtained in \cite{AP}. The random Schr\"odinger equation with a time-dependent potential
$$i \partial_t u +\frac 1 2 \Delta u -\varepsilon V(t,x) u=0$$
with the low frequency initial condition $u(0,x)=u_0(\varepsilon^\alpha x)$ for some $\alpha>0$ is treated in \cite{GR}, extending the homogenization result of Zhang and Bal (\cite{ZB1, ZB2}) for the case when $\alpha=1$ and $V(x)$ is a (not time-dependent) mean zero Gaussian random potential. Another related problem that has been extensively studied is the NLS with a time oscillating nonlinearity. In \cite{CS}, Cazenave and Scialom consider the NLS with time-oscillating nonlinearity
\begin{equation}
    \begin{cases} 
    i \partial_t u_\omega +\Delta u_\omega +\theta (\omega t) |u_\omega|^p u_\omega=0\notag\\
    u_\omega(0)=u_0 \in H^1(\R^d)
    \end{cases}
\end{equation}
where $p$ is an $H^1$-subcritical exponent and $\theta$ is a periodic function; they showed, firstly, that, as $|\omega|\to\infty$, the solution $u_\omega$ converges locally in time to soliton solutions of the stationary equation, obtained by the replacement of $\theta(t)$ with its time-average value, and secondly, that, if the limiting solution is global and has a certain decay property as $t\to\infty$, then $u_\omega$ is also global if $|\omega|$ is sufficiently large. Similar results for the critical problem where obtained by Fang and Han in \cite{FH}. The NLS with time-oscillating nonlinearity and dissipation
\begin{equation}
    \begin{cases} 
    i \partial_t u_\omega +\Delta u_\omega +\theta (\omega t) |u_\omega|^p u_\omega+ i \zeta (\omega t)u_\omega=0\notag\\
    u_\omega(0)=u_0 \in H^s(\R^d)
    \end{cases}
\end{equation}
for $0<s<\min\{1, \frac d 2\}$, $0<p<\frac{4}{d-2s}$ and $\theta, \zeta$ continuous periodic functions was studied in \cite{FZS}; it was shown that under some conditions, as $\omega\to \infty$, the solution will locally converge in Besov spaces to the solution of the homogenized equation $i u_t +\Delta u +\theta_0 |u|^p u+ i \zeta_0 u=0$ with the same initial condition, where $\theta_0$ and $\zeta_0$ are the average of $\theta$ and $\zeta$ respectively, and that if $\zeta_0$ is large enough, then the solution $u_\omega$ is global for sufficiently large $\omega$.

In this paper, we address the question of well-posedness for ($NLS_n$) and the behavior of the solutions as $n\to \infty$, which cannot be answered by any of the aforementioned results. We obtain sufficient conditions for the coupling function $g$ to ensure the existence and uniqueness of global solutions to ($NLS_n$) for $n$ sufficiently large, as well as a global in time homogenization result. More precisely, we show that, under these conditions, solutions to ($NLS_n$) converge to the solution to the homogeneous defocusing (NLS) in $L^4_{t,x}(\R\times \R^2)$. These results are recorded in the following theorem.

\begin{theorem} 
Let $u_0\in L^2(\R^2)$. Suppose $g\in L^\infty(\R^2)$, $\bar g\geq 0$ and for every $R>0$
\begin {equation}
\lim_{n\to \infty}\|(-\Delta +1)^{-1}(g(nx)-\bar g)\|_{L^\infty (|x|\leq R)}=0.
\end{equation}
Then for $n$ sufficiently large there exists a unique global solution $u_n$ to ($NLS_n$) with initial data $u_n(0)=u_0$; it scatters, in the sense that there exist $u_n^\pm \in L^2(\R^2)$ such that
$$\|u_n(t,x)-e^{it\Delta}u_{n}^{\pm}\|_{L^2(\R^2)}\to 0\quad\text{as}\quad t\to \pm\infty.$$
Moreover, if $u$ is the solution to (NLS) with initial data $u(0)=u_0$, then 
\begin{equation}
\lim_{n\to \infty}\|u_n-u \|_{L^4_{t,x}(\R\times \R^2)}=0.
\end{equation}

\end{theorem}

\begin{remark}
In fact, the solutions $u_n$ converge to $u$ in the Strichartz norm $\|\cdotp\|_S$, which will be defined in Section 2. 
\end{remark}

\begin{remark} One of the virtues of this result is that $g$ need not be a non-negative function. Our coupling function $g$ can take on a negative sign on subsets of $\R^2$, which suggests that we should be worried for focusing behavior. However, our result guarantees that homogenization arrests the dreaded blow-up.
\end{remark}

In the context of nonlinear optics that we briefly discussed above, this theorem implies that, under certain conditions and for high intensity laser beams, the propagation in an inhomogeneous medium approximates the propagation in a homogeneous Kerr medium, which is much better understood. It is worth noting that we only assume finite mass for the initial data; this is better fit to describe the optical power of a continuous wave laser beam. In general, we should not expect the propagation of an optical wave to preserve smoothness.

The idea for the proof of Theorem 1.1 is to approach the inhomogeneous problem as a perturbation of the homogeneous one for large values of $n$, which we know is globally well-posed with spacetime bounds, as demonstrated by Dodson. We should emphasize that a naive application of perturbation would not be effective, as there is no indication that $g-\bar g$ is small. Instead, we discover and exploit the non-resonant behavior of the coupling function, which leads to a more delicate perturbation argument. Unlike the usual homogenization problems, for which existence of global solutions is trivial and one only worries about convergence, in our case it is the homogenization that guarantees global well-posedness.

\begin{theorem}[Dodson, \cite{Dodson}]
For $u_0\in L^2(\R^2)$, there exists a unique global solution $u$ to (NLS) with initial data $u_0$ such that $$\|u\|_{L^4_{t,x}(\R\times \R^2)} \leq C(\|u_0\|_{L^2(\R^2)}).$$
The solution $u$ scatters, in the sense that there exist $u_\pm \in L^2(\R^2)$ such that
$$\|u(t,x)-e^{it\Delta}u_\pm\|_{L^2(\R^2)}\to 0\quad\text{as}\quad t\to \pm\infty.$$
Moreover, if $u_0\in H^s(\R^2)$ for $s>0$, then $u \in L^\infty_t H^s_x(\R \times \R^2)$.
\end{theorem}

For the mass-critical (NLS) we also have the following stability result. The norm $\|\cdotp \|_S$ here is again the Strichartz norm, which we will define later.

\begin{theorem}
[Stability for (NLS), \cite{TVZ}]
Let $I=[0,T]$ be a compact interval and let $v$ be an approximate solution to (NLS) in the sense that
$$(i \partial_t + \Delta )v = F(v) + e$$
for some function $e$, with initial data $v(0)=v_0$. Assume also that
\begin{align*}
    \|v\|_{L_t^\infty L_x^2(I\times \R^2)}&\leq M,\\
    \|v_0-u_0\|_{L^2(\R^2)}&\leq M',\\
    \|v\|_{L_{t,x}^4(I\times \R^2)}&\leq L,
\end{align*}
for some constants $M, M', L>0$, $u_0\in L^2(\R^2)$, and
\begin{align*}
    \|e^{it\Delta}(v_0-u_0)\|_{L_{t,x}^4(I\times \R^2)}&\leq \varepsilon,\\
    \|\int_0^t e^{i(t-s)\Delta}e(s)ds\|_{L_{t,x}^4(I\times \R^2)}&\leq \varepsilon
\end{align*}
for some $0<\varepsilon\leq \delta_0=\delta_0(M, M', L, \bar g)$. Then, there exists a solution $u$ to (NLS) on $I \times \R^2$ with initial data $u(0)=u_0$ satisfying
\begin{align*}
    \|u-v\|_{L_{t,x}^4(I\times \R^2)}&\leq C(M,M',L,\bar g)\varepsilon,\\
    \|u-v\|_{S(I\times \R^2)}&\leq C(M,M',L, \bar g)M',\\
    \|u\|_{S(I\times \R^2)}&\leq C(M,M',L,\bar g).
\end{align*}
\end{theorem}

With the goal we explained above in mind, we will start by investigating well-posedness and establish the perturbation theory for ($NLS_n$) in Section 3. In Section 4 we will show that in our setup, solutions to (NLS) are indeed approximate solutions to ($NLS_n$) for $n$ sufficiently large. Thus for $n$ sufficiently large the global well-posedness of the homogeneous problem gives rise to a unique global solution for \eqref{eqn:NLS_n}.

In Section 5 we showcase some especially interesting applications of Theorem 1.1. In particular, we show that our Theorem guarantees the existence of unique solution for large $n$ and homogenization whenever the coupling function $g$ is a trigonometric polynomial, a continuous quasi-periodic function or a bounded (not necessarily continuous) periodic function. Moreover, we show that our result can be applied to certain alloy-type models. These problems model disordered alloys in which the atoms of the various materials are located in lattice positions.

\subsection{A further question}

In the case when $\bar g=0$, Theorem 1.1 implies that the effects of the nonlinearity average to zero as $n$ becomes large. One may wonder whether there is a way to rescale the coefficient so that we get nontrivial effects. For instance, is there a parameter $\alpha>0$ so that the solutions to the equation
\begin{align}
i\partial_t u_n+\Delta u_n= n^\alpha g(nx) |u_n|^2 u_n
\end{align}
homogenize, but not to the linear Schr\"odinger equation? 

The first question we should ask is whether we are even able to construct solutions. As we explained earlier, one of the challenges of investigating homogenization for ($NLS_n$) was the fact that we do not even know whether global solutions exist. Our remedy in that case was to use a perturbation argument. However, the fact that the coupling functions $n^\alpha g(nx)$ are no longer bounded uniformly in $n$ means that the techniques we used to prove Theorem 1.1 cannot be modified to provide an answer for equation (1.3). This leads one to doubt even that global solutions exist for $n$ large. Below we outline a candidate for precisely the contrary. This suggests that, indeed, it is not merely a failure of our method, but rather an intrinsic problem. 

We will utilize existing blow-up results for inhomogeneous mass-critical NLS to construct solutions with arbitrarily small initial data that blow up arbitrarily fast. Let $g$ be a periodic function satisfying the hypotheses of our Theorem 1.1 with $\bar g=0$ that also obeys the conditions of Theorem 1.1 in \cite{BCD}. Then there exist $T>0$ and $\phi\in H^1$ such that the solution $v$ to 
$$i\partial_t v+\Delta v= g(x) |v|^2 v$$
with initial data $\phi$ blows up at time $T$. We consider 
$$v_n(t,x)= n^{1-\frac{\alpha}{2}} v(n^2 t,n x).$$
One can easily see that for each $n$, $v_n$ gives us a solution to (1.3) with initial data $v_n(0,x)= n^{1-\frac{\alpha}{2}} \phi(n x)$. Note that 
$$\|v_n(0)\|_{L^2(\R^2)}=n^{-\frac{\alpha}{2}}\|\phi\|_{L^2(\R^2)}\to 0$$
and that $v_n$ blows up at time 
$$T_n=n^{-2}T\to 0.$$
We believe that we can use this sequence of initial data to obtain for every $\varepsilon>0$ solutions $v_n$ to (1.3) with initial data $v_0$, $\|v_0\|_{L^2(\R^2)}<\varepsilon$, that blow up within time $t_n<\varepsilon$ for $n$ sufficiently large. The way we would imagine to do that is by translating these initial data by multiples of the period of $g$, far enough from each other to minimize their interaction, and `gluing' them together:
$$v_0(x)=\sum_{j\geq 1} \frac{1}{j^2} v_{n_j}(0, x-a_j) $$
for an appropriate subsequence $\{v_{n_j}(0)\}$ and translations $a_j\in L \mathbb Z^2$, where $L$ is the period of $g$.
This presents a compelling argument for the failure of the existence of solutions.

If we look at more regular initial data, maybe homogenization takes place for some $\alpha>0$. An attempt to modify the methods in the proof of Theorem 1.1 to attack this problem would require a suitable stability result. Unfortunately, the constants in such a stability result would have to depend on $n$ due to the lack of a uniform bound of the coupling functions, as we discussed above. As a result, even if we succeed in controlling the error (a quantity similar to (\ref{error}) in page \pageref{error}), we are unable to conclude that homogenization does or does not happen. For example, although it is not difficult to verify that 
$$\|\int_0^t e^{i(t-s)\Delta}n^\alpha g(nx) F(u (s))ds\|_{L^4_{t,x}(\R\times\R^2)}\lesssim n^{\alpha- \frac{2s}{s+2}}$$
where $u(t)=e^{it\Delta}u_0$ for initial data $u_0\in H^s(\R^2)$, i.e. the error up to which $u$ solves equation (1.3) with the same initial data converges to zero as $n\to\infty$ for $0<\alpha<2$ and $u_0$ sufficiently smooth depending on $\alpha$, we cannot proceed much further. 

We also observe the following interesting fact: If we consider a periodic coupling function $g$ satisfying the hypotheses of Theorem 1.1 for which solutions to equation (1.3) homogenize to (NLS) with coupling constant $\lambda$, it is possible to prove that the same holds (for the same constant $\lambda$) for any translation of $g$ by a rational multiple of its period. In particular, this implies that if (1.3) with $g(x)=\cos(x)$ homogenizes to (NLS) with coupling constant $\lambda$, so does (1.3) with $g(x)=-\cos(x)$. Since one would expect the latter to homogenize to (NLS) with coupling constant $-\lambda$, this fact leads us to believe that it is not likely that (1.3) homogenizes to (NLS) with $\lambda\neq 0$ for any $\alpha>0$. Hence, for more regular data maybe
(1.3) homogenizes to some more complicated equation, but not to (NLS).

\section*{Acknowledgements} I would like to thank my advisors, Rowan Killip and Monica Visan, for introducing me to this problem and for their invaluable support and guidance. I am also grateful to the anonymous referees for their careful reading of the manuscript and their
thoughtful comments. This work was supported in part by NSF grants DMS-1600942 (Rowan Killip) and DMS-1500707 (Monica Visan).

\section{Preliminaries}

We adopt the following convention for the Fourier transform:
$$\hat f(\xi)=\frac{1}{2\pi} \int_{\R^2} e^{-i \xi x} f(x)dx$$
for functions on the plane and
$$\hat f(n)=\frac{1}{2\pi} \int_{[0,2\pi)^2} e^{-i n x} f(x)dx$$
for functions on the torus $\R^2 /(2\pi \mathbb Z)^2$.

Throughout this paper we will denote the nonlinearities associated with $(NLS_n)$ and (NLS) by
$$F_n(u):=g (nx) \abs{u}^2 u \quad \text{and}\quad F(u):=\abs{u}^2 u.$$

We use the standard Littlewood-Paley operators $P_{\leq N}$, $P_{>N}$, given by
\begin{align*}
    \widehat{ P_{\leq N}f}(\xi)&:= m(\frac{\xi}{N}) \hat f(\xi),\\
    \widehat{ P_{> N}f}(\xi)&:= \big(1-m(\frac{\xi}{N})\big) \hat f(\xi)
\end{align*}
for $N\in 2^\mathbb Z$, where $m\in C_c^\infty(\R^2)$ is a radial bump function supported in the ball $\{\xi\in \R^2: |\xi| \leq 2\}$ and equal to 1 on the ball $\{\xi\in \R^2: |\xi| \leq 1\}$.
We will often refer to $P_{\leq N}$ and $P_{>N}$ as Littlewood-Paley projections onto low frequencies or high frequencies respectively (although they are not really projections). Like all Fourier multipliers, the Littlewood-Paley operators commute with the propagator $e^{it\Delta}$, as well as with differential operators. They also obey the following estimates.

\begin{lemma}
[Bernstein estimates]\cite{MS}
For $1\leq p\leq q\leq \infty$ and $s>0$,
\begin{align*}
    \|P_{N}f\|_{L^p}+\|P_{\leq N}f\|_{L^p}&\lesssim \|f\|_{L^p},\\
    \||\nabla|^{s} P_{\leq N}f\|_{L^p}&\lesssim N^s \|P_{\leq N}f\|_{L^p},\\
    \|P_{>N}f\|_{L^p}&\lesssim N^{-s} \||\nabla|^{s} P_{> N}f\|_{L^p},\\
    \|P_{\leq N}f\|_{L^q}&\lesssim N^{\frac 2 p -\frac 2 q} \|P_{\leq N}f\|_{L^p}.
\end{align*}

\end{lemma}

\begin{definition}
We say that the pair $(q, r)$ is Schr\"odinger admissible if $2\leq q, r \leq \infty$, $\frac {1}{q}+\frac{1}{r}=\frac {1}{2}$ and $(q, r)\neq (2, \infty)$. For a fixed spacetime slab $I\times \R^2$ we define the Strichartz norm
$$\|u\|_{S(I\times \R^2)}:=\sup_{(q,r) \,\text{admissible}}\|u\|_{L_t^q L_x^r (I\times \R^2)}.$$
We write $S(I\times \R^2)$ for the closure of all test functions under this norm.
\end{definition}

\begin{definition}[Solution]
Let $I$ be a compact interval containing zero. A function $u:I\times \R^2\to \mathbb C$ is a (strong) solution to ($NLS_n$) if it belongs to $C_t L^2_x \cap L_{t, \text{loc}}^4 L_x^4$ and obeys the Duhamel formula
$$u(t)= e^{it\Delta}u_0 -i \int_0^t e^{i(t-s)\Delta}F_n(u(s))ds$$
for all $t\in I$.
\end{definition}

\begin{lemma}
[Strichartz estimates] \cite{GV, Strichartz}
Let $(q, r)$ and $(\tilde q, \tilde r) $ be any Schr\"odinger admissible pairs. If $u$ solves 
\begin{equation}
    \begin{cases}
    i u_t +\Delta u=F\notag\\
    u(0)=u_0
    \end{cases}
\end{equation}
on $I\times \R^2$ for some time interval $I\ni 0$, then
$$\|u\|_{L_t^q L_x^r (I\times \R^2)}\lesssim \|u_0\|_{L_x^2(\R^2)}+\|F\|_{L_t^{\tilde q'} L_x^{\tilde r'} (I\times \R^2)},$$
where $\tilde q'$, $\tilde r'$ are the conjugate exponents of $\tilde q$, $\tilde r$ respectively ($\frac {1}{\tilde q}+\frac {1}{\tilde q'}=1=\frac {1}{\tilde r}+\frac {1}{\tilde r'}$).
\end{lemma}

\section {well-posedness and stability}

We start by showing well-posedness for ($NLS_n$).

\begin{theorem} [Well-posedness] 
Let $u_0\in L^2 (\R^2)$. There exists $\eta_0=\eta_0(\|g\|_{L^\infty(\R^2)})>0$ such that, for $0<\eta <\eta_0$ and $I$ compact interval containing zero satisfying 
\begin{equation}\label{eqn:a}
    \|e^{it\Delta}u_0\|_{L_{t,x}^4 (I\times \R^2)}\leq\eta,
\end{equation}
there exists a unique solution $u_n$ to ($NLS_n$) on $I\times \R^2$ for every $n$. Moreover, we have the following bounds

\begin{equation}
    \|u_n\|_{L_{t,x}^4 (I\times \R^2)}\leq 2\eta,
\end{equation}

\begin{equation}
    \|u_n\|_{S (I\times \R^2)}\lesssim \|u_0\|_{L_x^2(\R^2)}.
\end{equation}
\end{theorem}

\begin{proof} 
We use a contraction mapping argument.
We consider the solution map $u\mapsto \Phi_n(u)$
given by Duhamel's formula
$$\Phi_n(u)= e^{it\Delta}u_0 -i \int_0^t e^{i(t-s)\Delta}F_n(u(s))ds.$$
We will show that this is a contraction mapping on the set $\B=\B_1\cap \B_2$, where
$$\B_1= \{u\in L_t^\infty L_x^2(I\times \R^2) : \|u\|_{L_t^\infty L_x^2(I\times \R^2)} \leq 2\|u_0\|_{L^2(\R^2)}+C \|g\|_{L^\infty(\R^2)} (2\eta)^3\},$$
$$\B_2= \{u\in L_{t,x}^4(I\times \R^2) : \|u\|_{L_{t,x}^4(I\times \R^2)} \leq 2\eta\},$$
under the metric 
$$d(u,v)=\|u-v\|_{L_{t,x}^4(I\times \R^2)}.$$
Note that $C$ is the constant in Strichartz inequality. Also note that $\B_1, \B_2$ are closed, hence complete under $d$.

For $u\in \B$, Strichartz inequality yields
\begin{equation}
    \begin{split}
        \|\Phi_n(u)\|_{L_t^\infty L_x^2(I\times \R^2)}  &\leq \|u_0\|_{L^2(\R^2)}+ C\|F_n(u)\|_{L_{t,x}^\frac 4 3(I\times \R^2)}\notag\\
        &\leq \|u_0\|_{L^2(\R^2)}+ C \|g\|_{L^\infty(\R^2)} \|u\|_{L_{t,x}^4(I\times \R^2)}^3\\
        &\leq  2\|u_0\|_{L^2(\R^2)}+C \|g\|_{L^\infty(\R^2)} (2\eta)^3,
    \end{split}
\end{equation}
hence $\Phi_n(u)\in \B_1$. 

On the other hand, using Strichartz inequality and (3.1),
$$\|\Phi_n(u)\|_{L_{t,x}^4(I\times \R^2)} \leq \eta +C \|g\|_{L^\infty(\R^2)} \|u\|_{L_{t,x}^4(I\times \R^2)}^3 \leq (1+ 2C\|g\|_{L^\infty(\R^2)} (2\eta)^2)\eta \leq 2 \eta,$$
for $\eta>0$ sufficiently small, depending on $\|g\|_{L^\infty(\R^2)}$.

Therefore, if we choose $\eta_0$ sufficiently small so that $4\eta_0^2 C \|g\|_{L^\infty(\R^2)}<\frac 1 2$, $\Phi_n$ maps $\B$ to itself. Now we begin to show that it is indeed a contraction mapping.

First of all, note that for $u,v \in\B$
$$\abs{F_n(u)-F_n(v)}\leq \|g\|_{L^\infty(\R^2)}\abs{\abs{u}^2 u-\abs{v}^2v} \leq 2 \|g\|_{L^\infty(\R^2)} \abs{u-v} (\abs{u}^2 +\abs{v}^2).$$

Thus
\begin{equation}
    \begin{split}
        \|\Phi_n(u)-\Phi_n(v)&\|_{L_{t,x}^4(I\times \R^2)} \leq C \|F_n(u)-F_n(v)\|_{L_{t,x}^\frac 4 3(I\times \R^2)}\notag\\
        &\leq 2C \|g\|_{L^\infty(\R^2)} (\|u\|^2_{L_{t,x}^4(I\times \R^2)}+\|v\|^2_{L_{t,x}^4(I\times \R^2)})\|u-v\|_{L_{t,x}^4(I\times \R^2)}\\
        &\leq 4C \|g\|_{L^\infty(\R^2)} (2\eta)^2 \|u-v\|_{L_{t,x}^4(I\times \R^2)}.
    \end{split}
\end{equation}
By choosing $\eta_0$ even smaller if necessary, depending on $\|g\|_{L^\infty(\R^2)}$, we conclude that $\Phi_n$ is a contraction mapping. The fixed point theorem then guarantees the existence of a unique solution $u_n$ to ($NLS_n$). One more application of the  Strichartz inequality yields (3.3):
\begin{align*}
    \|u_n\|_{S(I\times \R^2)}&\leq C\|u_0\|_{L^2(\R^2)}+C\|g\|_{L^\infty(\R^2)}\|u_n\|_{L_{t,x}^4(I\times \R^2)}^2 \|u_n\|_{S(I\times \R^2)}\\
    &\leq C\|u_0\|_{L^2(\R^2)}+C\|g\|_{L^\infty(\R^2)}4 \eta^2 \|u_n\|_{S(I\times \R^2)}\\
    &\leq C\|u_0\|_{L^2(\R^2)}+\frac 1 2\|u_n\|_{S(I\times \R^2)},
\end{align*}
so
$$\|u_n\|_{S(I\times \R^2)}\leq 2 C\|u_0\|_{L^2(\R^2)}.$$
\end{proof}

\begin{remark}
Note that the constant $\eta_0$ does not depend on $n$, only on $\|g\|_{L^\infty(\R^2)}$. Therefore, if we are able to find a time interval $I$ satisfying the conditions of the theorem, we get local well-posedness for ($NLS_n$) for all $n$. Once again, $I$ will not depend on $n$, only on $u_0$ and $\|g\|_{L^\infty(\R^2)}$.

For small $L^2$-initial data, global well-posedness follows from combining Theorem 3.1 with the Strichartz inequality 
$$ \|e^{it\Delta}u_0\|_{L_{t,x}^4 (\R\times \R^2)}\lesssim \|u_0\|_{L^2(\R^2)}.$$

For general $L^2$-initial data, we are still able to prove the existence of a compact time interval $I$ such that \eqref{eqn:a} holds. Strichartz inequality allows us to apply the Dominated Convergence Theorem and show that 
$$\lim_{T\to0}\|e^{it\Delta}u_0\|_{L_{t,x}^4 ([-T, T]\times \R^2)}=0.$$
\end{remark}

Our next goal is to develop a stability result for the equation ($NLS_n$). In the next section, this will be used to compare solutions to ($NLS_n$) to solutions to the cubic NLS with a constant coupling constant; these solutions are known to be global and satisfy global spacetime bounds (Theorem 1.4 ).

The stability result adapted to ($NLS_n$) is modeled after the one for the mass-critical equation (Theorem 1.5). We present the details below.

\begin{comment}
In our problem, for fixed $u_0\in L^2(\R^2)$ and $g\in L^\infty (\R^2)$ let $\bar g$ be the nonnegative constant for which (1.1) holds. Then Dodson's result ensures the existence and uniqueness of a solution to \eqref{eqn:NLS} with coupling constant $\bar g$. In order to exploit this to obtain a solution for \eqref{eqn:NLS_n} we need to establish the necessary perturbation theory.
\end{comment}

\begin{lemma}[Short time perturbations] Fix $g\in L^\infty(\R^2)$. Let $I=[0,T]$ be a compact time interval and let $\tilde u$ be an approximate solution to \eqref{eqn:NLS_n} in the sense that 
$$i \tilde u_t +\Delta \tilde u=F_n(\tilde u)+e$$
for some function $e$, with initial data $\tilde u(0)=\tilde u_0\in L^2(\R^2)$. Assume also that
\begin{equation}
    \|\tilde u\|_{L_t^\infty L_x^2(I\times \R^2)}\leq M,
\end{equation}
\begin{equation}
    \|u_0-\tilde u_0\|_{L^2(\R^2)}\leq M'
\end{equation}
for some $M, M'>0$, $u_0\in L^2(\R^2)$, and
\begin{equation}
    \|\tilde u\|_{L_{t,x}^4(I\times \R^2)}\leq \varepsilon_0,
\end{equation}
\begin{equation}
    \|e^{it\Delta}(u_0-\tilde u_0)\|_{L_{t,x}^4(I\times \R^2)}\leq \varepsilon,
\end{equation}
\begin{equation}
    \|\int_0^t e^{i(t-s)\Delta}e(s) ds\|_{L_{t,x}^4 (I\times \R^2)}\leq \varepsilon
\end{equation}
for some $0<\varepsilon\leq \varepsilon_0=\varepsilon_0(M,M', \|g\|_{L^\infty(\R^2)})$ small. Then there exists a unique solution $u_n$ to \eqref{eqn:NLS_n} with initial data $u_n(0)=u_0$ satisfying
$$\|u_n-\tilde u \|_{L_{t,x}^4(I\times \R^2)}\lesssim \varepsilon,$$
$$\|u_n-\tilde u \|_{S(I\times \R^2)}\lesssim M',$$
$$\|u_n\|_{S(I\times \R^2)}\lesssim M+M',$$
$$\|F_n(u_n)-F_n(\tilde u)\|_{L_{t,x}^\frac4 3(I\times \R^2)}\lesssim\varepsilon,$$
where all implicit constants are allowed to depend on $\|g\|_{L^\infty(\R^2)}$.

\end{lemma}

\begin{proof}

Let $w=u_n-\tilde u$. Then $w$ is a solution to 
\begin{equation}
    \begin{cases}
    i w_t +\Delta w=F_n(\tilde u+w)-F_n(\tilde u)-e\\
    w(0)=w_0=u_0-\tilde u_0.
    \end{cases}
\end{equation}

For $t\in I$, we define $$A(t)=\|F_n(\tilde u+w)-F_n(\tilde u)\|_{L_{t,x}^\frac4 3([0,t]\times \R^2)}.$$
Then
\begin{equation}
    \begin{split}
        A(t)&\lesssim \|w(|\tilde u|^2+|w|^2)\|_{L_{t,x}^\frac4 3([0,t]\times \R^2)}\notag\\
        &\lesssim \|w\|_{L_{t,x}^4([0,t]\times \R^2)}^3+ \|\tilde u\|_{L_{t,x}^4([0,t]\times \R^2)}^2 \|w\|_{L_{t,x}^4([0,t]\times \R^2)}\\
        &\lesssim \|w\|_{L_{t,x}^4([0,t]\times \R^2)}^3 + \varepsilon_0^2 \|w\|_{L_{t,x}^4([0,t]\times \R^2)},
    \end{split}
\end{equation}
where the implicit constant depends only on $\|g\|_{L^\infty(\R^2)}$.

On the other hand, by Strichartz, (3.7) and (3.8),
$$\|w\|_{L_{t,x}^4([0,t]\times \R^2)}\lesssim \|e^{it\Delta}w_0\|_{L_{t,x}^4([0,t]\times \R^2)}+A(t)+\varepsilon\lesssim A(t)+\varepsilon.$$

So 
$$A(t)\lesssim (A(t)+\varepsilon)^3+ \varepsilon_0^2(A(t)+\varepsilon).$$
By taking cases based on which term of the right hand side dominates, we observe that, for $\varepsilon_0$ sufficiently small depending on $\|g\|_{L^\infty(\R^2)}$, the set $\{A(t):t\in [0,T]\}$ is a subset of $[0,\frac{1}{2}\varepsilon)\cup (\varepsilon_0,+\infty)$. However, $A(t)$ is continuous and $\lim_{t\to 0}A(t)=0$, so $A(t)\lesssim \varepsilon$ for all $t\in [0, T]$. Hence, 
\begin{align*}
    \|u_n-\tilde u \|_{L_{t,x}^4(I\times \R^2)}&=\|w\|_{L_{t,x}^4(I\times \R^2)}\lesssim A(T)+\varepsilon\lesssim\varepsilon,\\
    \|u_n-\tilde u \|_{S(I\times \R^2)}&\lesssim \|w_0\|_{L_x^2(\R^2)}+A(T)+\varepsilon \lesssim M'+\varepsilon\lesssim M'
\end{align*}
for $\varepsilon_0=\varepsilon_0(M', \|g\|_{L^\infty(\R^2)})$ sufficiently small.

By the Strichartz inequality,
\begin{equation}
    \begin{split}
        \|\tilde u\|_{S(I\times \R^2)}&\lesssim \|\tilde u\|_{L_t^\infty L_x^2(I\times \R^2)}+ \|\tilde u\|_{L_{t,x}^4(I\times \R^2)}^3+\varepsilon \notag\\
        &\lesssim M+\varepsilon_0^3+\varepsilon \lesssim M
    \end{split}
\end{equation}
for $\varepsilon_0=\varepsilon_0(M)$ small enough. Thus, for $\varepsilon_0=\varepsilon_0(M, M', \|g\|_{L^\infty(\R^2)})$ small enough,
$$\| u_n\|_{S(I\times \R^2)}\lesssim \|\tilde u\|_{S(I\times \R^2)}+M'\lesssim M+M'.$$
This completes the proof of the Lemma.
\end{proof}

\begin{theorem} [Stability]
Fix $g\in L^\infty(\R^2)$. Let $I=[0,T]$ be a compact time interval and let $\tilde u$ be an approximate solution to \eqref{eqn:NLS_n} in the sense that 
$$i \tilde u_t +\Delta \tilde u=F_n(\tilde u)+e$$
for some function $e$, with initial data $\tilde u(0)=\tilde u_0$. Assume also that
\begin{equation}
    \|\tilde u\|_{L_t^\infty L_x^2(I\times \R^2)}\leq M,
\end{equation}
\begin{equation}
    \|u_0-\tilde u_0\|_{L^2(\R^2)}\leq M',
\end{equation}
\begin{equation}
    \|\tilde u\|_{L_{t,x}^4(I\times \R^2)}\leq L
\end{equation}
for some $M, M', L >0$, $u_0\in L^2(\R^2)$, and
\begin{equation}
    \|e^{it\Delta}(u_0-\tilde u_0)\|_{L_{t,x}^4(I\times \R^2)}\leq \varepsilon,
\end{equation}
\begin{equation}
    \|\int_0^t e^{i(t-s)\Delta}e(s) ds\|_{L_{t,x}^4 (I\times \R^2)}\leq \varepsilon
\end{equation}
for some $0<\varepsilon\leq \varepsilon_1=\varepsilon_1(M,M',L, \|g\|_{L^\infty(\R^2)})$ small. Then there exists a unique solution $u_n$ to \eqref{eqn:NLS_n} with initial data $u_n(0)=u_0$ satisfying
\begin{equation}
 \|u_n-\tilde u \|_{L_{t,x}^4(I\times \R^2)}\leq C(M,M',L) \varepsilon,
 \end{equation}
\begin{equation}
\|u_n-\tilde u \|_{S(I\times \R^2)}\leq C(M,M',L) M',
\end{equation}
\begin{equation}
\|u_n\|_{S(I\times \R^2)}\leq C(M,M',L),
\end{equation}
where all implicit constants are allowed to depend on $\|g\|_{L^\infty(\R^2)}$.
\end{theorem} 

\begin{proof}

The assumptions of this theorem are very similar to those of Lemma 3.3; the only difference is that (3.6) has been replaced by (3.12). In the hope of being able to apply Lemma 3.3, we divide $I=[0,T]$ into $J$ subintervals $I_j=[t_j, t_{j+1}]$, $0\leq j<J$, so that
$$\|\tilde u\|_{L_{t,x}^4(I_j\times \R^2)}\leq \varepsilon_0=\varepsilon_0(M, 2M', \|g\|_{L^\infty(\R^2)})\,\,\text{for all}\,\,j.$$
Note that $J\sim (1+\frac L {\varepsilon_0})^4$. The reason we are using the small constant $\varepsilon_0$ associated with $2M'$ instead of $M'$ is that we cannot guarantee that $\|u_n(t_j)-\tilde u(t_j)\|_{L_x^2(\R^2)}\leq M'$, but we will prove that $\|u_n(t_j)-\tilde u(t_j)\|_{L_x^2(\R^2)}\leq 2M'$. \newline
Proceeding inductively, we show that for all $0\leq j<J$ and $0<\varepsilon\leq \varepsilon_1$, where $\varepsilon_1=\varepsilon_1(M,M',J)$ sufficiently small,
$$\|u_n-\tilde u \|_{L_{t,x}^4(I_j\times \R^2)}\leq C(j) \varepsilon,$$
$$\|u_n-\tilde u \|_{S(I_j\times \R^2)}\leq C(j) M',$$
$$\|u_n\|_{S(I_j\times \R^2)}\leq C(j)(M+M'),$$
$$\|F_n(u_n)-F_n(\tilde u) \|_{L_{t,x}^{\frac 4 3}(I_j\times \R^2)}\leq C(j) \varepsilon ,$$
where the constants $C(j)$ are allowed to depend on $\|g\|_{L^\infty(\R^2)}$.
Indeed, by Lemma 3.3, it suffices to show that 
$$\|u_n(t_j)-\tilde u(t_j)\|_{L_x^2(\R^2)}\leq 2 M'$$
and
$$\|e^{i(t-t_j)\Delta}(u_n(t_j)-\tilde u(t_j))\|_{L_{t,x}^4(I_j\times \R^2)}\lesssim \varepsilon $$
for all $0\leq j<J$ and $0<\varepsilon\leq \varepsilon_1.$

We verify these conditions by induction. Firstly, (3.11) and (3.13) suggest that the conditions clearly hold for $j=0$. Next, we assume the conditions are satisfied for all $0\leq k<j$. By the Strichartz inequality,
\begin{equation}
    \begin{split}
        \|u_n(t_j)-\tilde u(t_j)\|_{L^2(\R^2)} &\lesssim \|u_0-\tilde u_0\|_{L_x^2(\R^2)} +\|F_n(u_n)-F_n(\tilde u) \|_{L_{t,x}^{\frac 4 3}([0,t_j]\times \R^2)}+ \varepsilon\notag\\
        &\lesssim M' +\sum_{k=0}^{j-1} C(k)\varepsilon+\varepsilon.
    \end{split}
\end{equation}

Again by Strichartz inequality,
\begin{equation}
    \begin{split}
        \|e^{i(t-t_j)\Delta}(u_n(t_j)-\tilde u(t_j))\|_{L_{t,x}^4(I_j\times \R^2)}&\lesssim \|e^{it\Delta}(u_0-\tilde u_0)\|_{L_{t,x}^4(I_j\times \R^2)} \\
        &+ \|F_n(u_n)-F_n(\tilde u) \|_{L_{t,x}^{\frac 4 3}([0,t_j]\times \R^2)}+ \varepsilon\notag\\
        &\lesssim \varepsilon + \sum_{k=0}^{j-1} C(k)\varepsilon.
    \end{split}
\end{equation}
Thus, by choosing $\varepsilon_1=\varepsilon_1(M,M',J, \|g\|_{L^\infty(\R^2)})$ sufficiently small we can ensure that the conditions above will hold for $j$.
\end{proof}

\begin{remark}
At this point we would like to note that even for small values of $M'$, the constant $C(M,M',L)$ or the bound in (3.16) $C(M,M',L)M$ need not be small. In the proof of Lemma 3.3 we chose $\varepsilon_0<M'$. The argument in the proof of Theorem 3.4 suggests that the constant $C(M,M',L)$ depends on $J$, which becomes very large for small values of $M'$ ($J\gtrsim (M')^4$). However, if $\tilde u_0=u_0$, by applying the theorem above for $M'=1$, we are able to use Strichartz inequality and exploit the bounds (3.12) and (3.15) to obtain
\begin{align*}
    \|u_n-\tilde u\|_{S}&\lesssim \|u_n-\tilde u\|_{L_{t,x}^4}(\|\tilde u\|_{L_{t,x}^4}^2+\|u_n-\tilde u\|_{L_{t,x}^4}^2) +\|\int_0^t e^{i(t-s)\Delta}e(s) ds\|_{L_{t,x}^4 }\\
    &\lesssim C(M,L)\varepsilon (L^2+ C(M,L)^2\varepsilon^2)+\varepsilon\\
    &\leq \tilde C(M,L)\varepsilon.
\end{align*}

\end{remark}

\section{Proof of Theorem 1.1}

Fix $u_0\in L^2(\R^2)$. Let $g\in L^\infty(\R^2)$ and suppose there exists $\bar g\geq 0$ such that (1.1) holds for every $R>0$.
As we mentioned earlier, Dodson's result guarantees that there exists a unique global solution $u$ to (NLS) with initial data $u_0$. Our goal is to show that, for $n$ sufficiently large, $u$ is an approximate solution to ($NLS_n$) in the sense of Theorem 3.4. Then the perturbation theory we established in the previous section will imply that, for $n$ sufficiently large, ($NLS_n$) has a global solution that is unique, obeys spacetime bounds, and approximates $u$ in $L^4_{t,x}(\R\times \R^2)$.

Note that $u$ is a solution to 
\begin{equation}
\begin {cases}
i  u_t +\Delta u= F_n( u)+e_n\notag\\
u(0)=u_0
\end{cases}
\end{equation}
with
$$e_n=\bar g F(u)-F_n(u)=(\bar g-g(nx))F(u).$$ 
Clearly $u$ satisfies conditions (3.10)--(3.13) in Theorem 3.4, so it suffices to show that 
\begin{equation}
\tag{$\ast$}\label{error}
\Bigl\|\int_0^t e^{i(t-s)\Delta} (g(nx)-\bar g)  F(u (s)) ds\Bigr\|_{L_{t,x}^4(\R\times \R^2)}<\varepsilon
\end{equation}
for $n$ sufficiently large.

Before we begin, let us outline the steps we are going to follow and prove some useful results.

We will split $F(u)$ into a part where high frequencies of $u$ appear and a part including only low frequencies. For the first part, persistence of regularity and stability allow us to derive the desired bound provided high enough frequencies are present; the only property of $g$ that is required is boundedness. For the second part, the bound is obtained as a result of the smallness of $(-\Delta+1)^{-1}[g(nx)-\bar g]$ and $\nabla (-\Delta+1)^{-1}[g(nx)-\bar g]$ on a fixed ball for $n$ sufficiently large, combined with an estimate for the projection onto low frequencies of a smooth compactly supported function and the boundedness of $(-\Delta+1)^{-1}[g(nx)-\bar g]$ and $\nabla (-\Delta+1)^{-1}[g(nx)-\bar g]$ everywhere else. 
The following lemmas record the aforementioned estimates that will be essential in the proof of Theorem 1.1.

\begin{lemma}
Let $d\geq 1$. Suppose $f\in C_c^\infty(\R^d)$ and let $R>0$ such that $\supp f\subset B(0,R)$. Then for any $N\in 2^{\mathbb Z}$ such that $NR>1$, $1<p<\infty$, $x\in \R^d$ such that $|x|>2R$, and $c>0$,
$$|(P_{\leq N}f)(x)|\lesssim N^{\frac d p} \|f\|_{L^p(\R^d)} [(|x|-R)N]^{-c}.$$
\end{lemma}

\begin{proof}

Let $1<q<\infty$ such that $\frac 1 p +\frac 1 q=1$. In what follows $\check m$ denotes the inverse Fourier transform of $m$. Using H\"older inequality, the definition of the Littlewood-Paley projections and the fact that $\check m\in \mathcal S(\R^2)$, we get that 
\begin{equation}
    \begin{split}
        |(P_{\leq N}f)(x)|&\leq \int_{\{y\in\R^2: |y|\leq R\}} |f(y)|N^d |\check m (N(x-y))| dy \notag\\
        &\leq N^{\frac d p} \|f\|_{L^p(\R^d)} \big(\int_{\{y\in\R^2: |y-Nx|\leq NR\}} |\check m (y)|^{q} dy \big)^{\frac 1 q}\\
        &\lesssim N^{\frac d p} \|f\|_{L^p(\R^d)} \big(\int_{\{y\in\R^2: |y-Nx|\leq NR\}} \langle y\rangle ^{-(d+qc)} dy \big)^{\frac 1 q}\\
        & \lesssim N^{\frac d p} \|f\|_{L^p(\R^d)} [(|x|-R)N]^{-c}. 
    \end{split}
\end{equation}
\end{proof}

\begin{lemma}
Assume $g\in L^\infty(\R^2)$ and for every $R>0$
$$\lim_{n\to\infty}\|(-\Delta +1)^{-1}g(nx)\|_{L^\infty (|x|\leq R)}=0.$$ 
Then:
\begin{enumerate}
    \item 
    $\|(-\Delta +1)^{-1}g(nx)\|_{L^\infty (\R^2)}\lesssim \|g\|_{L^\infty(\R^2)}$  uniformly in $n$.\\[1mm]
    
    \item $\lim_{n\to \infty}\|\nabla(-\Delta +1)^{-1}g(nx)\|_{L^\infty (|x|\leq R)}=0$ for every $R>0$.\\[1mm]
    
    \item $\|\nabla(-\Delta +1)^{-1}g(nx)\|_{L^\infty (\R^2)}\lesssim \|g\|_{L^\infty(\R^2)}$ uniformly in $n$.
\end{enumerate}
\end{lemma}

\begin{proof}

Fix $M\in 2^{\mathbb N}$ and decompose $g(nx)=P_{>M}g(nx)+P_{\leq M}g(nx)$. One can easily see that 
$$-\Delta(-\Delta+1)^{-1}= 1-(-\Delta+1)^{-1},$$
hence by Bernstein
\begin{align*}
    \|(-\Delta +1)^{-1}P_{> M}&g(nx)\|_{L^\infty(\R^2)}\lesssim M^{-2} \|[1-(-\Delta +1)^{-1}]P_{> M}g(nx)\|_{L^\infty(\R^2)}\\
    &\lesssim M^{-2} \big(\|P_{> M}g(nx)\|_{L^\infty(\R^2)}+\|(-\Delta +1)^{-1}P_{> M}g(nx)\|_{L^\infty(\R^2)}\big).
\end{align*}
Then, for $M$ sufficiently large,
\begin{equation}
\|(-\Delta +1)^{-1}P_{> M}g(nx)\|_{L^\infty(\R^2)}\lesssim M^{-2} \|P_{> M}g(nx)\|_{L^\infty(\R^2)}\lesssim M^{-2}\|g\|_{L^\infty(\R^2)}
\end{equation}
uniformly in $n$.
    
To estimate the low frequencies, let
$$K_1(x)=  \Big[\frac{1}{|\xi|^2+1} m(\tfrac{\xi}{M})\Big]\check (x)$$
where $m\in C_c^\infty(\R^2)$ is the smooth cutoff function associated with the Littlewood-Paley projection.

Observe that for any $h$,
$$[(-\Delta +1)^{-1}P_{\leq M}h](x)=[K_1\ast h](x).$$
Since $g\in L^\infty(\R^2)$ and $K_1\in L^1(\R^2)$ as the inverse Fourier transform of a Schwartz function, we conclude that
\begin{equation}
\|(-\Delta +1)^{-1}P_{\leq M}g(nx)\|_{L^\infty(\R^2)}\lesssim \|g\|_{L^\infty(\R^2)}\quad\text{uniformly in}\,\, n.
\end{equation}
Combining (4.1) and (4.2) we obtain 
\begin{equation}
\|(-\Delta +1)^{-1}g(nx)\|_{L^\infty(\R^2)}\lesssim \|g\|_{L^\infty(\R^2)}\quad\text{uniformly in}\,\, n.
\end{equation}
    
Next, let $h_n(x)=(-\Delta +1)^{-1}g(nx)$. To estimate the high frequencies we exploit the boundedness of $g$. Using Bernstein and (4.1),
\begin{align}
    \|\nabla P_{> M}h_n(x)\|_{L^\infty(\R^2)}&\lesssim M^{-1} \|\Delta(-\Delta+1)^{-1} P_{> M}g(nx)\|_{L^\infty(\R^2)}\notag\\
    &\lesssim M^{-1} \big(\|P_{> M}g(nx)\|_{L^\infty(\R^2)}+ M^{-2} \|P_{> M}g(nx)\|_{L^\infty(\R^2)}\big)\notag\\
    &\lesssim M^{-1}\|g\|_{L^\infty(\R^2)}\quad\text{uniformly in}\quad n.
\end{align}

Let $R>0$, $\varepsilon>0$. There exists $n_0\in \N$ such that for all $n>n_0$
$$\|h_n(x)\|_{L^\infty (|x|\leq 2R)}<\varepsilon.$$
Let
$$K_2(x)= [\xi m(\tfrac{\xi}{M})]\check  (x).$$
Then clearly 
$$\|K_2\|_{L^1(\R^2)}\lesssim M.$$
For $|x|\leq R$, using (4.3) and the rapid decay of $K_2$,
\begin{align}
        |\nabla P_{\leq M}h_n(x)|&\leq \int_{\R^2} |h_n(y)| |K_2(x-y)|dy\notag\\
        &\leq\int_{\{|x-y|\leq R\}} \varepsilon |K_2(x-y)|dy + \int_{\{|x-y|> R\}} |h_n(y)| |K_2(x-y)|dy\notag\\
        &\lesssim_M \varepsilon + \|g\|_{L^\infty(\R^2)} \int_{\{|y|> R\}} \langle y \rangle ^{-4} dy\notag\\
        &\lesssim_M \varepsilon +\|g\|_{L^\infty(\R^2)} R^{-2}\notag\\
        &\lesssim_M \varepsilon
\end{align}
by taking $R$ sufficiently large. Combining (4.4) and (4.5) and letting $\varepsilon\to 0$, $R\to\infty$, and then $M\to\infty$ we derive claim (2).

We now turn to claim (3). Observe that for any $h$
$$\nabla (-\Delta +1)^{-1}P_{\leq M}h=K_3\ast h,$$
where 
$$K_3(x)= \big[\frac{\xi}{|\xi|^2+1} m(\tfrac{\xi}{M})\big]\check  (x).$$
Then, since $K_3\in L^1(\R^2)$ (as the inverse Fourier transform of a Schwartz function), and $g\in L^\infty (\R^2)$, we conclude that 
\begin{equation}
\|\nabla(-\Delta +1)^{-1}P_{\leq M}g(nx)\|_{L^\infty (\R^2)}\lesssim \|g\|_{L^\infty(\R^2)}
\end{equation}
uniformly in $n$. Then (4.4) and (4.6) imply that
\begin{equation}
\|\nabla(-\Delta +1)^{-1}g(nx)\|_{L^\infty (\R^2)}\lesssim \|g\|_{L^\infty(\R^2)}
\end{equation}
uniformly in $n$.

\end{proof}

\begin{remark}
Note that the estimates (4.1), (4.2), (4.4) and (4.6), and consequently (4.3) and (4.7), hold for any $g\in L^\infty(\R^2)$.
\end{remark}

We are now ready to complete the proof of Theorem 1.1.

\begin{proof}[Proof of Theorem 1.1]
As stated above, it is enough to show that given $\varepsilon>0$,
$$\Bigl\|\int_0^t e^{i(t-s)\Delta} (g(nx)-\bar g)  F(u)(s) ds\Bigr\|_{L_{t,x}^4(\R\times \R^2)}<\varepsilon$$
for $n$ sufficiently large. We observe that for $h(x)=g(x)-\bar g$ we can write $g(x)=h(x)+\bar g$ and $h$ satisfies (1.1) with $\bar h=0$. Therefore, without loss of generality, we may assume $\bar g=0$ in what follows whenever the quantity $g(nx)-\bar g$ is involved.

We split $u$ into low and high frequencies, $P_{\leq N} u$ and $P_{> N} u$ respectively. Here $N$ is a large dyadic integer that will be chosen shortly. 
%For the terms where high frequencies appear, the only property of $g$ that we exploit is boundedness. A sufficiently large choice of $N$ will give us the desired bound uniformly in $n$. Having already chosen $N$, in order to make the terms where only low frequencies appear small, we will require that $n$ is sufficiently large.

First we deal with the terms where at least one $P_{>N} u$ is present. For this we will only use the boundedness of $g$. In what follows, we will allow the implicit constants to depend on $\|u_0\|_{L^2(\R^2)}$ and on $\|g\|_{L^\infty(\R^2)}$.

Let $\delta_0=\delta_0(\|u_0\|_{L^2(\R^2)},\|u_0\|_{L^2(\R^2)},  C(\|u_0\|_{L^2(\R^2)}), \bar g)$ be the small constant in Theorem 1.5, where $C(\|u_0\|_{L^2(\R^2)})$ is the spacetime bound in Dodson's result (Theorem 1.4). We choose $0<\delta\ll \min\{\varepsilon, \delta_0\}$ and $M\in 2^{\mathbb Z}$ large such that 
$$\|P_{>M}u_0\|_{L^2(\R^2)}<\delta.$$
Take $v_0=u_0-P_{>M}  u_0= P_{\leq M} u_0$. By Dodson's result, there exists a global solution $v$ to (NLS) with initial data $v_0$ which satisfies 
$$\|v\|_{L_{t,x}^4(\R\times\R^2)}\leq C(\|u_0\|_{L^2(\R^2)}).$$ 
Persistence of regularity, combined with Strichartz, guarantees that 
$$\|\nabla v\|_{S(\R\times\R^2)} \lesssim \|\nabla v_0\|_{L^2(\R^2)} \lesssim M.$$
Since $\|P_{>M} u_0\|_{L^2(\R^2)}<\delta$, by Strichartz we have that $\|e^{it\Delta} (u_0 -v_0)\|_{L_{t,x}^4(\R\times \R^2)}\lesssim\delta$. Then Theorem 1.5 ensures that 
$\| u-v\|_{L_{t,x}^4(\R\times\R^2)}\lesssim \delta$ and so
\begin{equation}
\begin {split}
\|P_{>N}  u\|_{L_{t,x}^4(\R\times\R^2)}&\leq \|P_{>N} ( u-v)\|_{L_{t,x}^4(\R\times\R^2)}+\|P_{>N} v\|_{L_{t,x}^4(\R\times\R^2)}\notag\\
&\lesssim \delta+N^{-1} \|\nabla v\|_{L_{t,x}^4(\R\times\R^2)}\lesssim \delta
\end{split}
\end{equation}
for $N\gg \frac M \delta$.

Therefore, for $N$ sufficiently large the previous estimate and Strichartz inequality yield
\begin{align*}
\|\int_0^t e^{i(t-s)\Delta} g(nx) H_N(u(s)) ds\|_{L_{t,x}^4(\R\times \R^2)}\lesssim \|P_{>N} u\|_{L_{t,x}^4(\R\times\R^2)} \| u\|_{L_{t,x}^4(\R\times\R^2)}^2 \lesssim \varepsilon,
\end{align*}
where $H_N(u)=F(u)-F(P_{\leq N}u)= P_{>N}u(|u|^2+\bar u P_{\leq N}u)+\overline{P_{>N}u}(P_{\leq N}u)^2$.

Next, we turn our attention to the terms where only $P_{\leq N} u$ appears. Having fixed $N\in\N$ large, we will show that 
$$\Bigl\|\int_0^t e^{i(t-s)\Delta} g(nx)F(P_{\leq N} u)(s) ds\Bigr\|_{L_{t,x}^4(\R\times \R^2)}<\varepsilon$$
for $n$ sufficiently large.

One can easily see that for functions $F$ and $G$ we have the following identity:

\begin{equation}
    \begin{split}
        (-\Delta +1)^{-1}(FG)= F (-\Delta +1)^{-1}G&+ (-\Delta +1)^{-1}(\Delta F (-\Delta +1)^{-1} G) \\
        &+2 (-\Delta +1)^{-1} (\nabla F  \cdotp \nabla (-\Delta +1)^{-1}G).
    \end{split}
\end{equation}

Integrating by parts and using (4.8), we obtain that

$$\Bigl|\int_0^t e^{i(t-s)\Delta}g(nx) F(P_{\leq N} u)(s)ds\Bigr|\lesssim |I_1|+|I_2|+|I_3|+|I_4|+ |B_1|+|B_2|$$
where

$$I_1=\int_0^t e^{i(t-s)\Delta}[(-\Delta +1)^{-1}g(nx)] F(P_{\leq N} u)(s)ds,$$ 
$$I_2=\int_0^t e^{i(t-s)\Delta}[(-\Delta +1)^{-1}g(nx)] \frac{d}{ds}F(P_{\leq N} u)(s)ds,$$
$$I_3=\int_0^t e^{i(t-s)\Delta}[(-\Delta +1)^{-1}g(nx)] \Delta F(P_{\leq N} u)(s)ds,$$
$$I_4=\int_0^t e^{i(t-s)\Delta} [\nabla(-\Delta +1)^{-1}g(nx)]\cdotp \nabla F(P_{\leq N} u)(s)ds,$$
$$B_1=F(P_{\leq N}\ u)(t)[(-\Delta +1)^{-1}g(nx)],$$
$$B_2=e^{it\Delta}F(P_{\leq N} u)(0)[(-\Delta +1)^{-1}g(nx)].$$

We will show that the $L^4_{t,x}$ norm of each one of these terms is less than $\varepsilon$ for $n$ sufficiently large. Let $\eta>0$ to be chosen later. The main tools we have at our disposal are Lemma 4.1 and Lemma 4.2. In order to take advantage of Lemma 4.1, we need to approximate $u$ by a compactly supported function $v$ in some appropriate space.

The analysis of each term is slightly different, but the idea behind all of them can be summarized as follows:  First of all, we split the term we are working with into a term that contains $v$ and one that contains $u-v$. The first term can be estimated using the smallness of $(-\Delta +1)^{-1}g(nx)$ on a fixed ball containing the support of $v$ and the smallness of $P_{\leq N}v$ outside it. The essential ingredient for the second term is that $v$ is an approximation of $u$, combined with the boundedness of $(-\Delta +1)^{-1}g(nx)$.

We take $v\in C_c^\infty(\R\times \R^2)$ and $v_0\in C_c^\infty (\R^2)$ such that 
\begin{equation}
\| u-v\|_{L^4_{t,x}(\R\times \R^2)}<\eta,\quad \|u_0-v_0\|_{L^2(\R^2)}<\eta.
\end{equation}
This is possible because $u\in L_{t,x}^4 (\R\times \R^2)$ and $u_0\in L^2(\R^2)$.
Let $R>1$ be such that $\supp v\subset B_R(0)$ and $\supp v_0\subset B_R(0)$ (note that the first ball is in $\R\times \R^2$ and the second in $\R^2$). By Lemma 4.2, there exists $n_0\in \N$ such that for all $n>n_0$ 

\begin {equation}
\|(-\Delta +1)^{-1}g(nx)\|_{L^\infty (|x|\leq 3R)}<\eta\quad \text{and} \quad \|\nabla(-\Delta +1)^{-1}g(nx)\|_{L^\infty (|x|\leq 3R)}<\eta.
\end{equation}
Moreover, we know that
\begin {align}
\|(-\Delta +1)^{-1}g(nx)\|_{L^\infty(\R^2)}+ \|\nabla(-\Delta +1)^{-1}g(nx)\|_{L^\infty(\R^2)}\lesssim\|g\|_{L^\infty(\R^2)}
\end{align}
where the implicit constant does not depend on $n$.

Let's begin with $I_1$. We consider $n>n_0$. By Strichartz,
\begin{equation}
\begin{split}
\|I_1\|_{L_{t,x}^4(\R\times \R^2)}& \leq \|[(-\Delta +1)^{-1}g(nx)] F(P_{\leq N} u)(t,x)\|_{L_{t,x}^\frac 4 3(\R\times \R^2)}\notag\\
&\leq \|I_1'\|_{L_{t,x}^\frac 4 3(\R\times \R^2)}+\|I_1''\|_{L_{t,x}^\frac 4 3(\R\times \R^2)},
\end{split}
\end{equation}
where 
$$I_1'=[(-\Delta +1)^{-1}g(nx)] F(P_{\leq N}v)(t,x),$$
$$I_1''=[(-\Delta +1)^{-1}g(nx)] (F(P_{\leq N} u)(t,x)-F(P_{\leq N}v)(t,x)).$$
Then by H\"older and (4.10)
$$\|I_1'\|_{L_{t,x}^\frac 4 3(|x|\leq 3R)}\lesssim \eta \|P_{\leq N}v\|_{L_{t,x}^4(\R\times \R^2)}^3\lesssim \eta (\|u\|_{L_{t,x}^4}+\eta)^3.$$
An application of Lemma 4.1 for $c=5$ gives us the estimates
\begin {equation}
\|P_{\leq N}v\|_{L_{t,x}^4(|x|>3R)}+\|P_{\leq N}v\|_{L_{t}^4L_x^{12}(|x|>3R)}\lesssim N^{-\frac 9 2} \|v\|_{L_{t,x}^4}\lesssim N^{-\frac 9 2} (\| u\|_{L_{t,x}^4}+\eta).
\end {equation}
Then, by (4.11) and (4.12),
$$\|I_1'\|_{L_{t,x}^\frac 4 3(|x|> 3R)}\lesssim \|P_{\leq N}v\|_{L_{t,x}^4(|x|>R)}^3\lesssim N^{-\frac{27}{2}}\|v\|_{L_{t,x}^4(\R\times \R^2)}^3\lesssim N^{-\frac{27}{2}}(\|u\|_{L_{t,x}^4} +\eta)^3.$$
Finally by H\"older, (4.9) and (4.11)
$$\|I_1''\|_{L_{t,x}^\frac 4 3}\lesssim \| u-v\|_{L_{t,x}^4(\R\times \R^2)} (\| u\|_{L_{t,x}^4(\R\times \R^2)}^2+\|v\|_{L_{t,x}^4(\R\times \R^2)}^2)\lesssim \eta (\| u\|_{L_{t,x}^4} +\eta)^2.$$

Next, we consider $I_2$. Note that, since $ u$ is a solution to (NLS), Strichartz inequality yields
\begin{equation}
    \begin{split}
        \|I_2\|_{L_{t,x}^4(\R\times \R^2)}&\leq \|[(-\Delta +1)^{-1}g(nx)] \frac{d}{dt} F(P_{\leq N} u)(t,x)\|_{L_{t,x}^\frac 4 3(\R\times \R^2)}\notag\\
        &\lesssim\|I_5\|_{L_{t,x}^\frac 4 3(\R\times \R^2)}+\|I_6\|_{L_{t,x}^\frac 4 3(\R\times \R^2)},
    \end{split}
\end{equation}
where
$$I_5=[(-\Delta +1)^{-1}g(nx)] \Delta P_{\leq N} u | P_{\leq N} u|^2,$$
$$I_6=[(-\Delta +1)^{-1}g(nx)] P_{\leq N}(|u|^2  u) | P_{\leq N} u|^2.$$

Starting with $I_5$, we observe that
\begin{equation}
\|\Delta P_{\leq N} u\|_{L_{t,x}^4(\R\times \R^2)}\lesssim N^2 \|   u\|_{L_{t,x}^4(\R\times \R^2)}.
\end {equation}
We consider $I_5'= [(-\Delta +1)^{-1}g(nx)] \Delta P_{\leq N} u | P_{\leq N}v|^2$ and $I_5''= I_5-I_5'$.
Then, by H\"older, (4.9), (4.10), (4.11), (4.12) and (4.13),

$$\|I_5'\|_{L_{t,x}^\frac 4 3(|x|\leq 3R)} \lesssim \eta \|\Delta P_{\leq N} u \|_{L_{t,x}^4} \|P_{\leq N}v\|_{L_{t,x}^4}^2\lesssim \eta N^2 \| u \|_{L_{t,x}^4}(\| u \|_{L_{t,x}^4}+\eta)^2,$$

$$\|I_5'\|_{L_{t,x}^\frac 4 3(|x|> 3R)} \lesssim N^2 \| u \|_{L_{t,x}^4} \|P_{\leq N}v \|_{L_{t,x}^4(|x|>3R)}^2 \lesssim N^{-7} \| u \|_{L_{t,x}^4} (\| u \|_{L_{t,x}^4} +\eta)^2.$$
On the other hand, H\"older inequality and the estimates (4.9), (4.11) and (4.13) give 
$$\|I_5''\|_{L_{t,x}^\frac 4 3} \lesssim N^2 \| u \|_{L_{t,x}^4} \| u-v \|_{L_{t,x}^4} (\| u \|_{L_{t,x}^4}+ \|v \|_{L_{t,x}^4}) \lesssim \eta N^2 \| u \|_{L_{t,x}^4}  (\| u \|_{L_{t,x}^4} +\eta).$$

To estimate $I_6$, we take advantage of the fact that 
\begin {align}
\|P_{\leq N}(| u|^2 u) P_{\leq N} u\|_{L_{t,x}^2(\R\times \R^2)}&\lesssim  \|P_{\leq N}(| u|^2  u)\|_{L_{t}^2 L_x^{4}(\R\times \R^2)} \|P_{\leq N}  u\|_{L_{t}^\infty L_x^{4}(\R\times \R^2)}\notag\\
&\lesssim N^{\frac{3}{2}} \|| u|^2  u\|_{L_{t}^2 L_x^{1}(\R\times \R^2)} N^{\frac{1}{2}} \| u\|_{L_{t}^\infty L_x^{2}(\R\times \R^2)}\notag\\
&\lesssim N^2 \| u\|_{L_{t}^{6} L_x^{3}(\R\times \R^2)}^3 \| u\|_{L_{t}^{\infty} L_x^{2}(\R\times \R^2)}.
\end{align}
We decompose into $I_6'= [(-\Delta +1)^{-1}g(nx)]  P_{\leq N}(| u|^2  u) | P_{\leq N} u| | P_{\leq N}v|$ and $I_6''= I_6-I_6'$. Working similarly to what we did previously, using H\"older and the estimates (4.9), (4.10), (4.11), (4.12) and (4.14), we obtain

$$\|I_6'\|_{L_{t,x}^\frac 4 3(|x|\leq 3R)} \lesssim \eta N^2 \| u\|_{L_{t}^{6} L_x^{3}}^3 \|u\|_{L_{t}^{\infty} L_x^{2}}(\| u \|_{L_{t,x}^4}+\eta)^,$$

$$\|I_6'\|_{L_{t,x}^\frac 4 3(|x|> 3R)} \lesssim N^{-\frac 5 2} \| u\|_{L_{t}^{6} L_x^{3}}^3 \| u\|_{L_{t}^{\infty} L_x^{2}} (\| u \|_{L_{t,x}^4} +\eta),$$

$$\|I_6''\|_{L_{t,x}^\frac 4 3} \lesssim \eta N^2 \| u\|_{L_{t}^{6} L_x^{3}}^3 \| u\|_{L_{t}^{\infty} L_x^{2}}.$$

For $I_3$, we have that
\begin{equation}
    \begin{split}
        \|I_3\|_{L_{t,x}^4}&\leq \|[(-\Delta +1)^{-1}g(nx)] \Delta F(P_{\leq N} u)(t,x)\|_{L_{t,x}^\frac 4 3}\notag\\
        &\lesssim\|I_5\|_{L_{t,x}^\frac 4 3}+\|I_7\|_{L_{t,x}^\frac 4 3},
    \end{split}
\end{equation}
where
$$I_7=[(-\Delta +1)^{-1}g(nx)] P_{\leq N} u | \nabla P_{\leq N} u|^2.$$
We have already estimated $I_5$ and we will treat $I_7$ similarly. Once again, we consider $I_7'=[(-\Delta +1)^{-1}g(nx)]  | \nabla P_{\leq N} u|^2 P_{\leq N}v $ and $I_7''= I_7-I_7'$. H\"older inequality and Bernstein, combined with the estimates (4.9), (4.10), (4.11) and (4.12) yield
$$\|I_7'\|_{L_{t,x}^\frac 4 3(|x|\leq 3R)} \lesssim \eta \|\nabla P_{\leq N} u \|_{L_{t,x}^4}^2 \|P_{\leq N}v\|_{L_{t,x}^4}\lesssim \eta N^2 \| u \|_{L_{t,x}^4}^2 (\| u \|_{L_{t,x}^4}+\eta),$$

$$\|I_7'\|_{L_{t,x}^\frac 4 3(|x|> 3R)} \lesssim N^2 \|u \|_{L_{t,x}^4}^2 \|P_{\leq N}v \|_{L_{t,x}^4(|x|>3R)} \lesssim N^{-\frac 5 2} \| u \|_{L_{t,x}^4}^2 (\| u \|_{L_{t,x}^4} +\eta),$$

$$\|I_7''\|_{L_{t,x}^\frac 4 3} \lesssim N^2 \| u \|_{L_{t,x}^4}^2 \| u-v \|_{L_{t,x}^4} \lesssim \eta N^2 \| u \|_{L_{t,x}^4}^2 .$$

To estimate $I_4$, we use Strichartz inequality once again and perform a decomposition similar to what we did for the previous terms. More precisely,
$$\|I_4\|_{L_{t,x}^4} \lesssim \|[\nabla (-\Delta +1)^{-1}g(nx)] \cdotp \nabla P_{\leq N} u |P_{\leq N} u |^2\|_{L_{t,x}^\frac 4 3} \leq \|I_4'\|_{L_{t,x}^\frac 4 3}+\|I_4''\|_{L_{t,x}^\frac 4 3},$$
where 
$$I_4'=[\nabla(-\Delta +1)^{-1}g(nx)]\cdotp \nabla P_{\leq N} u (t,x) P_{\leq N} u (t,x) P_{\leq N}v(t,x) ,$$
$$I_4''=[\nabla (-\Delta +1)^{-1}g(nx)] \cdotp \nabla P_{\leq N} u (t,x) P_{\leq N} u (t,x) [P_{\leq N} u(t,x)-P_{\leq N}v(t,x)].$$
Then, using H\"older inequality, Bernstein and the estimates (4.9), (4.10), (4.11), (4.12) we get
$$\|I_4'\|_{L_{t,x}^\frac 4 3(|x|\leq 3R)} \lesssim \eta \|\nabla P_{\leq N} u \|_{L_{t,x}^4} \| P_{\leq N} u \|_{L_{t,x}^4} \|v\|_{L_{t,x}^4}\lesssim \eta N \|u \|_{L_{t,x}^4}^2 (\| u \|_{L_{t,x}^4}+\eta),$$

$$\|I_4'\|_{L_{t,x}^\frac 4 3(|x|> 3R)} \lesssim N \| u \|_{L_{t,x}^4}^2 \|P_{\leq N}v \|_{L_{t,x}^4(|x|>3R)} \lesssim N^{-\frac 7 2} \| u \|_{L_{t,x}^4}^2 (\|u \|_{L_{t,x}^4} +\eta),$$

$$\|I_4''\|_{L_{t,x}^\frac 4 3} \lesssim N \| u \|_{L_{t,x}^4}^2 \| u-v \|_{L_{t,x}^4}  \lesssim \eta N \| u \|_{L_{t,x}^4}^2 .$$

Next, we turn our attention to the boundary terms. We start by decomposing $B_1=B_1'+B_1''$, where
$$B_1'=[(-\Delta +1)^{-1}g(nx)]|(P_{\leq N} u)(t)|^2 (P_{\leq N}v)(t) ,$$
$$B_1''=[(-\Delta +1)^{-1}g(nx)]|(P_{\leq N}u)(t)|^2 [(P_{\leq N} u)(t)-(P_{\leq N}v)(t)].$$
We apply H\"older inequality and use Bernstein and the estimates (4.9), (4.10), (4.11), (4.12) to get
$$\|B_1'\|_{L_{t,x}^4(|x|\leq 3R)}\lesssim \eta \|P_{\leq N} v\|_{L_t^{4} L_x^{12}} \|P_{\leq N}  u\|_{L_t^{\infty} L_x^{12}}^2 \lesssim \eta N^2  \| u\|_{L_t^{\infty} L_x^{2}}^2 (\|u\|_{L_{t,x}^4}+\eta),$$
\begin{align*}
\|B_1'\|_{L_{t,x}^4(|x|> 3R)}&\lesssim \|P_{\leq N} v\|_{L_t^{4} L_x^{12}(|x|> 3R)} \|P_{\leq N} u\|_{L_t^{\infty} L_x^{12}}^2\\
&\lesssim N^{-\frac {17}{6}}  \|u\|_{L_t^{\infty} L_x^{2}}^2 (\| u\|_{L_{t,x}^4}+\eta),
\end{align*}
$$\|B_1''\|_{L_{t,x}^4}\lesssim \|P_{\leq N} u- P_{\leq N} v\|_{L_t^{4} L_x^{12}} \|P_{\leq N} u\|_{L_t^{\infty} L_x^{12}}^2 \lesssim \eta N^2  \| u\|_{L_t^{\infty} L_x^{2}}^2 .$$

For $B_2$, Strichartz inequality yields
$$\|B_2\|_{L^4_{t,x}}\lesssim \|F(P_{\leq N}  u_0)(-\Delta +1)^{-1}g(nx)\|_{L^2_{x}(\R^2)}\leq \|B_2'\|_{L^2_{x}(\R^2)}+\|B_2''\|_{L^2_{x}(\R^2)},$$
where
$$B_2'=[(-\Delta +1)^{-1}g(nx)]F(P_{\leq N}v_0),$$
$$B_2''=[(-\Delta +1)^{-1}g(nx)](F(P_{\leq N} u_0)-F(P_{\leq N}v_0)).$$
Then, by (4.10), Bernstein and (4.9),
$$\|B_2'\|_{L^2(|x|\leq 3R)}\lesssim\eta \|P_{\leq N}v_0\|_{L^6}^3\lesssim \eta N^2 (\|u_0\|_{L^2}+\eta)^3.$$
We can use Lemma 4.1 for $c=5$ to see that
$$\|P_{\leq N}v_0\|_{L^6(|x|>3R)}\lesssim N^{-4} \|v_0\|_{L^2}\lesssim N^{-4} (\|u_0\|_{L^2}+\eta).$$
Using (4.11) and the estimate above, we get
$$\|B_2'\|_{L^2(|x|> 3R)}\lesssim \|P_{\leq N}v_0\|_{L^6(|x|>3R)}^3\lesssim  N^{-12} (\|u_0\|_{L^2}+\eta)^3.$$
Finally, by H\"older, Bernstein and (4.9),
\begin{align*}
    \|B_2''\|_{L^2}&\lesssim \|P_{\leq N}u_0- P_{\leq N}v_0\|_{L^6} ( \|P_{\leq N}u_0\|_{L^6}^2 +\|P_{\leq N}v_0\|_{L^6}^2)\\
    &\lesssim  N^{2} \|u_0- v_0\|_{L^2} ( \|u_0\|_{L^2}^2 +\|v_0\|_{L^2}^2)\\
    &\lesssim \eta N^2 (\|u_0\|_{L^2}+\eta)^2.
\end{align*}

All in all, by taking $N$ even larger if necessary, so that $N^{-1}\ll\varepsilon$, we choose $\eta>0$ sufficiently small so that $\eta N^2\ll\varepsilon$. This choice guarantees that 
$$\Bigl\|\int_0^t e^{i(t-s)\Delta} g(nx)F(P_{\leq N} u)(s) ds\Bigr\|_{L_{t,x}^4(\R\times \R^2)}<\varepsilon$$
for $n>n_0$ and thus completes the proof.
\end{proof}

\section{Applications}

As we stated in the Introduction, there are several interesting examples that satisfy the hypotheses of Theorem 1.1. Below we discuss them individually; however, as the following lemma shows, they may be combined.

\begin{lemma}
Convex combinations of functions for which the conditions of Theorem 1.1 hold, also satisfy these conditions.
\end{lemma}

\begin{proof}
It suffices to show it for a convex combination of two functions. 

Let $g_1, g_2\in L^\infty(\R^2)$ and suppose there exist nonnegative constants $\bar g_1, \bar g_2$ such that for every $R>0$
$$\lim_{n\to \infty}\|(-\Delta +1)^{-1}(g_1(nx)-\bar g_1)\|_{L^\infty (|x|\leq R)}=0$$
and
$$\lim_{n\to \infty}\|(-\Delta +1)^{-1}(g_2(nx)-\bar g_2)\|_{L^\infty (|x|\leq R)}=0.$$
Let $0<\lambda<1$ and $g=\lambda g_1 + (1-\lambda) g_2$, $\bar g=\lambda \bar g_1 + (1-\lambda) \bar g_2$. Clearly $g\in L^\infty(\R^2)$ and $\bar g \geq0$. Moreover, for every $x\in \R^2$
\begin{equation}
\begin{split}
|(-\Delta +1)^{-1}(g(nx)-\bar g)|\leq \lambda & |(-\Delta +1)^{-1}(g_1(nx)-\bar g_1)|\notag\\
&+ (1-\lambda) |(-\Delta +1)^{-1}(g_2(nx)-\bar g_2)|
\end{split}
\end{equation}
so 
$$\lim_{n\to \infty}\|(-\Delta +1)^{-1}(g(nx)-\bar g)\|_{L^\infty (|x|\leq R)}=0.$$
\end{proof}

\subsection{Trigonometric Polynomials}
We will show that if $g$ is a trigonometric polynomial, i.e. 
$$g(x)=\sum_{k\in \mathbb Z^2}c_k e^{i k \cdotp x}$$
with $c_k\in \mathbb C$ for every $k\in \mathbb Z^2$ and only finitely many of them are nonzero, and in addition $c_0\geq 0$, then $g$ satisfies the conditions of Theorem 1.1. 

It is obvious that $g\in L^\infty(\R^2)$. The natural choice for $\bar g$ here is $\bar g=c_0\geq 0$. We only need to prove that $g$ satisfies (1.1). It is enough to prove it in the case when $g$ is a character, i.e. 
$$g(x)=e^{i k\cdotp x}$$
for some $k\in\mathbb Z^2$.

For $k=0$, one can see that (1.1) is trivially true since $g(nx)-\bar g=0$.

For $k\neq 0$, we get $\bar g=0$. Moreover,
$$(-\Delta +1)^{-1} g(nx) =\frac {1}{n^2|k|^2 +1}e^{in k\cdotp x},$$
so
$$\lim_{n\to\infty} \|(-\Delta +1)^{-1} g(nx)\|_{L^\infty(\R^2)}=\lim_{n\to\infty}\frac {1}{n^2|k|^2 +1}=0. $$

\subsection{Continuous (quasi-)periodic functions}

Consider $G:\R^d\to \R$ a $2\pi$-periodic continuous function such that $\hat G(0)\geq 0$ and $A$ a $d\times 2$ matrix whose rows are linearly independent over $\mathbb Z$. Then the function $g:\R^2\to \R$ given by
$$g(x)=G(Ax)$$
is quasi-periodic. Note that this also covers the case where $g$ is periodic with respect to some (not necessarily rectangular) lattice. We will show that $g$ satisfies the assumptions of Theorem 1.1.

We choose $\bar g= \hat G(0)$; we may assume that $\bar g= \hat G(0)=0$. 

Let $\varepsilon>0$. Since $G$ is continuous, it can be approximated in $L^\infty(\R^2)$ by trigonometric polynomials; there exists trigonometric polynomial $f$ such that 
$$\|G-f\|_{L^\infty(\R^d)}<\frac {\varepsilon} {(2\pi)^d}.$$
By (4.3), 
$$\|(-\Delta +1)^{-1} (g(nx)-(f\circ A)(nx))\|_{L^\infty(\R^2)}\lesssim \|g-f\circ A\|_{L^\infty(\R^2)}\lesssim \varepsilon$$
uniformly in $n$.

On the other hand, $f$ is a trigonometric polynomial, so $f(x)=\sum_{k\in \mathbb Z^2}c_k e^{ik\cdotp x}$, where $c_k\in \mathbb C$ for all $k\in \mathbb Z^2$ and all but finitely many of them are zero. By the way $f$ was defined, we have that $|c_0|<\frac \varepsilon 2$. Then
$$(-\Delta +1)^{-1} (f\circ A)(nx) =c_0 + \sum_{k\in\mathbb Z^d\setminus \{0\}} c_k \frac {1}{n^2|kA|^2 +1}e^{in k\cdotp A x},$$
so for $n$ sufficiently large
$$\|(-\Delta +1)^{-1} (f\circ A)(nx)\|_{L^\infty(\R^2)}\leq \frac \varepsilon 2 +\frac {1}{n^2} \sum_{k\in\mathbb Z^d\setminus \{0\}}  \frac {|c_k|}{|kA|^2} <\varepsilon.$$
Note that $|kA|> 0$ for all $k\in\mathbb Z^d\setminus \{0\}$ and only finitely many of them are present in our sum since $c_k=0$ for all but finitely many $k$, so $\min_{k\in\mathbb Z^d\setminus \{0\}} |kA|>0$.

Combining these two estimates we obtain that
$$\lim_{n\to\infty} \|(-\Delta +1)^{-1} g(nx)\|_{L^\infty(\R^2)}=0. $$

\subsection{Bounded periodic functions}
Let $g\in L^\infty([0, 2\pi)^2)$ with $\hat g(0)\geq 0$. We choose $\bar g=\hat g(0)$. Once again, we can assume that $\bar g=\hat g(0)=0$. Then
$$(-\Delta +1)^{-1} g(nx) =\sum_{k\in\mathbb Z^2\setminus \{0\}} \hat g(k) \frac {1}{n^2|k|^2 +1}e^{in k\cdotp x},$$
so by Cauchy-Schwarz
$$\|(-\Delta +1)^{-1} g(nx)\|_{L^\infty(\R^2)} \leq \|\hat g\|_{l^2 (\mathbb Z^2)} \frac {1}{n^2} \Big(\sum_{k\in\mathbb Z^2\setminus \{0\}}  \frac {1}{|k|^4}\Big)^{\frac 12}\lesssim \|g\|_{L^\infty([0, 2\pi)^2)} \frac {1}{n^2} $$
and we can conclude that 
$$\lim_{n\to\infty} \|(-\Delta +1)^{-1} g(nx)\|_{L^\infty(\R^2)}=0. $$

\remark In the above we worked out the case that $g$ is $2\pi$-periodic. However, the same argument can be modified to show that the result holds for functions periodic with respect to any lattice.

\remark Note that in this case $g$ is not necessarily continuous, unlike the previous examples.

\remark In the three previous examples we obtained an estimate stronger than (1.1). Of course it is very advantageous that Theorem 1.1 requires only a weaker assumption, as there are cases where this stronger estimate is not available but (1.1) still holds; the following example falls in this case.

\subsection{Bernoulli alloy-type model}

Let $\phi:\R^2\to \mathbb C$ such that $\langle x\rangle^{2+\epsilon} \phi \in L^\infty (\R^2)$ for some $\epsilon>0$. Also consider $X_k$, $k\in \mathbb Z^2$, independent identically distributed Bernoulli random variables that take the value 1 with probability $\frac 1 2$ and the value -1 with probability $\frac 1 2$. We are interested in the function
$$g(x)=\sum_{k\in \mathbb Z^2}X_k \phi(x-k).$$
We will show that $g$ satisfies the hypotheses of Theorem 1.1 almost surely. 

First of all, it is easy to see that $g\in L^\infty (\R^2)$. Since $\langle x\rangle^{2+\epsilon} \phi \in L^\infty (\R^2)$, 
$$|g(x)|\leq \sum_{k\in \mathbb Z^2}|\phi(x-k)|\lesssim 1$$ for every $x\in \R^2$.

The natural choice of $\bar g$ here is $\bar g=0$. Note that for every $x\in \R^2$
$$\E[g(x)]= \E[\sum_{k\in \mathbb Z^2}X_k  \phi(x-k)]=0.$$

Showing that $g$ satisfies (1.1) requires more work. Once again, it is convenient to decompose in high and low frequencies. For high frequencies, the proof of Lemma 4.2 supplies us with useful estimates, the majority of which hold for all bounded functions as we remarked earlier. For low frequencies, we have the added advantage that the integral kernel associated with the operator $(-\Delta+1)^{-1}P_{\leq N}$ is Schwartz. 

Let $\varepsilon>0$. Since $g\in L^\infty (\R^2)$, we get (4.1) uniformly in $n$, so we can fix $N\in 2^\N$ large enough so that 
$$\|(-\Delta +1)^{-1} P_{>N} g(nx)\|_{L^{\infty}(\R^2)}<\varepsilon$$
for all $n\in\N$.

Having fixed $N$ large, we turn our attention to $(-\Delta +1)^{-1} P_{\leq N} g(nx)$. Let $R\in \N$ and consider the ball $B_R:=\{x\in \R^2: |x|\leq R\}$. Let $n\in\N$ and consider squares of the form $S^m:= \frac m n + [0, \frac 1 n)^2$ for $m\in \mathbb Z^2$. Observe that $B_R$ can be covered by $O(n^2 R^2)$ many of these squares, say $B_R\subset \bigcup_{m\in I} S^m$ for some $I\subset \mathbb Z$ with $|I|=O(n^2 R^2)$. We will first show that, for $n$ sufficiently large, $|(-\Delta +1)^{-1} P_{\leq N} g(nx)|$ is small at the bottom left corners of our little squares, i.e. at the points $\frac m n$ for $m\in I$, and then that the difference of the values of $(-\Delta +1)^{-1} P_{\leq N} g(nx)$ between the bottom left corner and any other point of the square $S^m$ is small.

We begin by estimating the function at the bottom left corners. In the following, we denote by $K$ the integral kernel that satisfies $(-\Delta +1)^{-1}P_{\leq N} f(x)=(K\ast f)(x)$.

Fix $x\in \R^2$ and $u,v \in \R^2$. Then using Riemann sums one can see that
\begin{equation}
    \lim_{n\to \infty} \frac {1}{n^2} \sum_{k\in \mathbb Z^2} K(x-\frac{u+k}{n}) \overline{K(x-\frac{v+k}{n})}= \int_{\R^2}|K(x-z)|^2 dz= \int_{\R^2}|K(z)|^2 dz.
\end{equation}
Then for each $k\in \mathbb Z^2$ a change of variables gives
$$| \int_{\R^2} K(x-y) \phi(ny-k) dy|^2=\frac{1}{n^4} \int_{\R^2}\int_{\R^2} K(x-\frac{u+k}{n}) \overline{K(x-\frac{v+k}{n})} \phi(u) \overline{\phi(v)} du dv,$$
so
\begin{equation}
    \lim_{n\to \infty} n^2 \sum_{k\in \mathbb Z^2} | \int_{\R^2} K(x-y) \phi(ny-k) dy|^2 =  |\int_{\R^2}\phi(z) dz|^2 \int_{\R^2}|K(z)|^2 dz.
\end{equation}
The interchange of integration and the infinite sum is justified by the Monotone Convergence Theorem, and the Dominated Convergence Theorem allows us to use (5.1). This result suggests that there exists $n_0\in \N$ such that for all $n>n_0$ and for all $x\in \R^2$
\begin{equation}
    n^2 \sum_{k\in \mathbb Z^2} | \int_{\R^2} K(x-y) \phi(ny-k) dy|^2 <2 |\int_{\R^2}\phi(z) dz|^2 \int_{\R^2}|K(z)|^2 dz.
\end{equation}

Now we fix $m\in I$ and consider the square $S^m$. Let $x_m=\frac m n$ be the bottom left corner point of this square. We want to calculate $\E[|(-\Delta +1)^{-1} P_{\leq N} g(n \cdot)(x_m)|^4]$. Since the random variables $X_k$, $k\in\mathbb Z^2$, are independent with mean 0 and variance 1, 
\begin{equation}
    \begin{split}
        \E[|(-\Delta +1)^{-1} P_{\leq N} g(n \cdot)(x_m)|^4]&= \big( \sum_{k\in \mathbb Z^2} | \int_{\R^2} K(x_m-y) \phi(ny-k) dy|^2 \big) ^2\notag\\
        &\leq 4 \frac {1}{n^4} |\int_{\R^2}\phi(z) dz|^4 \big( \int_{\R^2}|K(z)|^2 dz \big)^2.
    \end{split}
\end{equation}
Therefore
\begin{equation}
    \begin{split}
        \E[\sup _{m\in I}|(-\Delta +1)^{-1} P_{\leq N} g(n \cdot)(x_m)|^4]&\lesssim |I| \frac {1}{n^4} |\int_{\R^2}\phi(z) dz|^4 \big( \int_{\R^2}|K(z)|^2 dz \big)^2\\
        &=O(n^{-2}).
    \end{split}
\end{equation}
This implies that
$$\mathbb P[\sup_{m\in I}|(-\Delta +1)^{-1} P_{\leq N} g(n \cdot)(x_m)|\geq \varepsilon] \lesssim \frac {1}{\varepsilon^4} n^{- 2},$$
so
$$\sum_{n>n_0} \mathbb P[\sup_{m\in I}|(-\Delta +1)^{-1} P_{\leq N} g(n \cdot) (x_m)|\geq \varepsilon] <\infty.$$
Then by Borel-Cantelli we conclude that almost surely 
\begin{equation}
\sup_{m\in I}|(-\Delta +1)^{-1} P_{\leq N} g(n \cdot)( x_m)|<\varepsilon
\end{equation}
for all but finitely many $n\in \N$.

We are left to estimate the difference between the values of $(-\Delta +1)^{-1} P_{\leq N} g(n \cdot)$ at the bottom left corner of $S^m$ and any other point of the square. 

Fix $x_1, x_2\in \R^2$ and $u \in \R^2$. Using Riemann sums once again we conclude that
\begin{equation}
    \lim_{n\to \infty} \frac {1}{n^2} \sum_{k\in \mathbb Z^2} \big| K(x_1-\frac{u+k}{n})- K(x_2-\frac{u+k}{n}) \big|= |x_1-x_2|\int_{\R^2}|\nabla K(z)|dz.
\end{equation}

Arguing as earlier, 
\begin{equation}
    \lim_{n\to \infty}\sum_{k\in \mathbb Z^2}  \int_{\R^2} \big|K(x_1-y)-K(x_2-y)\big| | \phi(ny-k)| dy= |x_1-x_2| \int_{\R^2}|\phi|  \int_{\R^2}|\nabla K| 
\end{equation}
and consequently 
\begin{equation}
    \sum_{k\in \mathbb Z^2}  \int_{\R^2} \big|K(x_1-y)-K(x_2-y)\big| | \phi(ny-k)| dy<2 |x_1-x_2| \int_{\R^2}|\phi|  \int_{\R^2}|\nabla K| 
\end{equation}
for all $n>n_0$ for some $n_0\in \N$ and for all $x_1, x_2\in \R^2$.

Fix $m\in I$ and consider the points of the square $S^m$, $x_m=\frac m n$ and $x$. Since $|X_k|=1$ with probability 1, the absolute value of the difference between the values of $(-\Delta +1)^{-1} P_{\leq N} g(n \cdot)$ at $x_m$ and $x$ is bounded by the left hand side of (5.8) with $x_1, x_2$ replaced by $x_m, x \in S^m$. Recall that $S^m$ is a square of side length $\frac 1 n$, therefore 
\begin{equation}
    \sup_{x\in S^m}|(-\Delta +1)^{-1} P_{\leq N} g(n \cdot)(x_m)-(-\Delta +1)^{-1} P_{\leq N} g(n \cdot)(x)|
    <4 \frac{1}{n} \int_{\R^2}|\phi|  \int_{\R^2}|\nabla K| .
\end{equation}

Then
\begin{equation}
    \begin{split}
        \E[\sup_{x\in S^m}|(-\Delta +1)^{-1} P_{\leq N} g(n \cdot)(x_m)-(-\Delta +1)^{-1} P_{\leq N} g(n \cdot)(x)|^4]
        \lesssim \frac {1}{n^4} 
    \end{split}
\end{equation}
and the same argument as before suggests that almost surely 
\begin{equation}
\sup_{m\in I} \sup_{x\in S^m}|(-\Delta +1)^{-1} P_{\leq N} g(n \cdot)(x_m)-(-\Delta +1)^{-1} P_{\leq N} g(n \cdot)(x)|<\varepsilon
\end{equation}
for all but finitely many $n\in \N$.

Combining (5.5) and (5.11) we conclude that almost surely for all but finitely many $n\in \N$
\begin{equation}
\sup_{x\in B_R}|(-\Delta +1)^{-1} P_{\leq N} g(n \cdot)(x)|<\varepsilon.
\end{equation}

\begin{remark}
The only properties of $X_k$ used in the above were that they are i.i.d. bounded random variables with $\E[X]=0$, so the same argument can be applied to any alloy-type model with these properties. Moreover, the result extends for any $X_k$ i.i.d. bounded random variables with $\E[X]\geq 0$ and $\int\phi \geq 0$. This can be seen by splitting $g=g_1+g_2$ where
$$g_1(x):=\sum_{k\in \mathbb Z^2}(X_k-\E[X])\phi(x-k),$$ $$g_2(x):=\E[X] \sum_{k\in \mathbb Z^2}\phi(x-k).$$ 
The fact that $X_k-\E[X]$ are mean zero i.i.d. random variables ensures that $g_1$ satisfies the conditions of Theorem 1.1, and $g_2$ is bounded periodic with $\int g_2\geq 0$. Then Lemma 5.1 implies that $g$ satisfies the conditions of Theorem 1.1.
\end{remark}

\end{document}